\documentclass[reqno,12pt]{amsart}

\title[Fast Tensor Transforms]{Fast Tensor Transforms and Ring-Valued
  Orthogonal Matrices: An Application to Cryptography}

\author{Jacques Peyri\`ere} \address{Institut de Math\'ematiques
  d'Orsay, CNRS, Universit\'e Paris-Saclay, 91405
  Orsay, France.} \email{jacques.peyriere@universite-paris-saclay.fr} 
\email{peyriere@phare.normalesup.org}



\usepackage{amsmath,amsfonts,amssymb,mathrsfs} 
\newcommand*{\alphabet}{{\mathscr A}}
\newcommand*{\Id}{\mathrm{Id}}
\newcommand*{\mat}[1]{{\mathsf #1}}
\newcommand{\radem}{{\mathsf r}}
\newcommand*{\rev}[1]{{\overleftarrow{#1}}}
\newcommand*{\ring}{{\mathscr R}}
\newcommand{\walsh}{{\mathsf w}}

\newcommand*{\trsp}[1]{{}^{\mathsf t}#1}
\newtheorem{theorem}{Theorem}
\newtheorem{lemma}[theorem]{Lemma}
\theoremstyle{definition}
\newtheorem{remark}[theorem]{Remark}
\newtheorem*{remark*}{Remark}
\begin{document}

\begin{abstract}
  We present a generalization of the Fast Fourier and fast Walsh
  transform algorithms to tensor products of $n$ arbitrary $q\times q$
  matrices over a commutative ring, reducing the cost of applying such
  a tensor product from $q^{2n}$ to $nq^{n}$ ring operations. We then
  give an explicit construction of orthogonal matrices over a
  commutative unitary ring, starting from a prescribed first row, and
  specialize this construction to $\ring={\mathbb Z}/256{\mathbb Z}$,
  counting the number of admissible parameter choices for each matrix
  size~$q$. Combining these two ingredients, we propose a symmetric
  cryptosystem in which the secret key determines~$n$ orthogonal
  matrices over ${\mathbb Z}/256{\mathbb Z}$ whose tensor product
  encrypts a block of $q^n$ bytes; encryption and decryption both
  exploit the fast tensor product algorithm, reducing the per-block
  cost from $q^{2n}$ to $nq^n$ byte multiplications. We tabulate key
  size, block size, and computational cost for practical parameter
  ranges $2\le q\le 12$, $2\le n\le 6$, and give three examples of
  implementations.
\end{abstract}

\maketitle

\section{Introduction}

In the first section, we revisit the well-known Fast Fourier Transform
algorithm and its analogue for the Walsh transform.

The second section deals with a completely different topic. We give a
construction of square matrices $\mat{M}$ on a unitary commutative
ring such that $\trsp{\mat{M}}\mat{M}=\Id$, where $\Id$ is the unit
matrix of suitable dimension and $\trsp{\mat{M}}$ stands for the
transpose of~$\mat{M}$ (of course we will call them orthogonal
matrices). The particular case of ${\mathbb Z}/256{\mathbb Z}$ is
treated.

In the third section, we propose a symmetric cryptosystem based on
these considerations and give three examples of implementations.

\section{Fast Tensor Product}

Let~$q$ be an integer larger than 1 and
$\alphabet = \{0,1,2,\dots,q-1\}$. Let
$ \alphabet^* = \displaystyle\bigcup_{n\ge 0} \alphabet^n$ stand for
the set of words on the alphabet~$\alphabet$ (as usual, the only
element of $\alphabet^0$ is $\epsilon$, the empty word). The set of
nonempty words on the alphabet~$\alphabet$ is denoted by
$\alphabet^+$.

We number the letters of a nonempty word from~0, starting from the
left. If $\xi=\xi_0\xi_1\dots\xi_{n-1}\in \alphabet^+$, the length~$n$
of this word is denoted by~$|\xi|$, and the reverse word
$\xi_{n-1}\xi_{n-2}\dots\xi_{0}$ will be denoted as $\rev{\xi}
$.

We also denote the concatenation of words simply by juxtaposition:
$vw=v_0v_1\cdots v_{|v|-1}w_0w_1\cdots w_{|w|-1}$.\medskip

Let $\mat{R}_1,\mat{R}_2,\cdots,\mat{R}_{n}$ be square matrices with
coefficients in a commutative ring~$\ring$, indexed by
$\alphabet\times\alphabet$. More precisely
$$ \mat{R}_\nu =  \left(\radem_{\nu,i}^j\right)_{(i,j)\in\alphabet^2} \text{\quad for\quad} 1\le \nu \le n
$$

When $1\le m\le n$, $v=v_1v_2\cdots v_m\in \alphabet^m$ and
$w=w_1w_2\cdots w_m\in \alphabet^m$ we set
$$\walsh_v^w = \prod_{j=1}^{m} \radem_{j,v_j}^{w_j}.$$
Then the matrix $R=\left(\walsh_v^w\right)_{v,w\in
  \alphabet^n}$ is the tensor product $\mat{R}_1\otimes
\mat{R}_2\otimes\cdots\otimes \mat{R}_n$.

We shall freely consider the rows of matrices as functions. This
means that we shall consider $\walsh_v^w$ and
$\walsh_v(w)$ as synonyms.

\begin{lemma}[Fast Tensor Product] Let
  $X=\left(X(w)\right)_{w\in \alphabet^n}$ be a vector, also
  considered as a function on $\alphabet^n$. Define a sequence of
  vectors in the following way.
  \begin{enumerate} \item $X_0 =X$,
  \item
    $\displaystyle X_{m}(viw) = \sum_{j\in\alphabet}
    \radem_{m,i}^{j}X_{m-1}(jvw)$ for $1\le m\le n$, $i\in \alphabet$,
    $|v|= n-m$, and $|w|=m-1$.
  \end{enumerate}
  Then, for $0\le m\le n$,
  \begin{equation}\label{f1}
  X_m(v\rev{w}) =
  \sum_{|u|=m} \walsh_w^uX_0(uv)\text{ for } |v|=n-m \text{ and } |w|=m.
\end{equation}
In particular
$\displaystyle X_n(\rev{w}) = \sum_{|u|=n} \walsh_w^u X(u) \text{
  for } |w|=n.$
\end{lemma}

\proof Formula~\eqref{f1} holds for $m=1$. Suppose it holds for
some~$m$. Then, for $i\in \alphabet$, $w\in \alphabet^m$ and
$v\in \alphabet^{n-m-1}$,
\begin{eqnarray*}
  X_{m+1}(v\rev{wi}) &=& \sum_{j\in\alphabet}\radem_{m+1,i}^{j} X_{m}(jv\rev{w})\\
                     &=& \sum_{j\in\alphabet}\sum_{u\in \alphabet^m}
                         \walsh_w^u\radem_{m+1,i}^{j}X_0(ujv)\\
                     &=&  \sum_{i=0}^{q-1}\sum_{u\in \alphabet^m} \walsh_{wi}^{uj}X_0(ujv)\\
  &=& \sum_{u\in \alphabet^{m+1}} \walsh_{wi}^u X_0(uv).
\end{eqnarray*}
So Formula~\eqref{f1} holds for $m+1$.\medskip

This lemma shows that, up to a permutation of its coefficients, $X_n$
equals the product $\mat{R}X$. The advantage is that this computation
requires only~$nq^n$ multiplications, whereas the naive computation
requires~$q^{2n}$ multiplications. This is in the same spirit as the
fast Fourier transform~\cite{Cooley}, the fast Walsh
transform~\cite{Yates}, and related algorithms~\cite{ahmed,blahut,winograd}.

\section{Construction of some orthogonal matrices on a ring}
\subsection{Orthogonal matrices whose first row is given}

Let $q$ be an integer larger than~2 and $(a_j)_{0\le j<q}$ be real
numbers such that $0< |a_0|< 1$ and $\sum a_j^2= 1$. Consider the
row-matrix $L=(a_1\ a_2\ \cdots\ a_{q-1})$.

We are looking for a vector $\mat{V}$ and for a square
matrix~$\mat{A}$ such that
$$\mat{U} = 
\begin{pmatrix}
a_0&\mat{L}\\\mat{V}&\mat{A}
\end{pmatrix}
$$
is an orthogonal matrix.

This means that we must have
\begin{eqnarray}
  &&a_0^2+\trsp{\mat{V}}\mat{V} = 1,\label{e1}\\ &&
                                               a_0\mat{L}+\trsp{\mat{V}}\mat{A}=0,\label{e2}\\ && \trsp{\mat{L}}\mat{L}
                                                                                             +\trsp{\mat{A}}\mat{A} = \mat{I}_{q-1}.\label{e3}
\end{eqnarray}

Observe that $\mat{L}\trsp{\mat{L}}= 1-a_0^2$ and
$\bigl( \trsp{\mat{L}}\mat{L}\bigr)^2 =
(1-a_0^2)\,\trsp{\mat{L}}\mat{L}$. So we have
$\left(\mat{I}-\trsp{\mat{L}}\mat{L}\right) \bigl(\mat{I} +
a_0^{-2}\,\trsp{\mat{L}}\mat{L}\bigr)= \mat{I}$, which shows that any
solution to~\eqref{e3} is invertible.

Suppose \eqref{e2} and~\eqref{e3} are fulfilled. Then
\begin{eqnarray*}
\trsp{\mat{V}}\mat{V} &=& a_0^2\mat{L}\bigl(
\trsp{\mat{A}}\mat{A}\bigr)^{-1}\,\trsp{\mat{L}}\\
                      &=& a_0^2\mat{L}(1+a_0^{-2}\,\trsp{\mat{L}}\mat{L})\trsp{\mat{L}}\\
&=& a_0^2\bigl((1-a_0^2)+a_0^{-2}(1-a_0^2)^2\bigr)\\
                 &=& 1-a_0^{2}.
\end{eqnarray*}
This means that~\eqref{e1} is also fulfilled.

If $\mat{A}$ and $\mat{B}$ are solutions to~\eqref{e3}, we have
$\trsp{\mat{A}}\mat{A}=\trsp{\mat{B}}\mat{B}$,\\ i.e.,
$\trsp{\bigl(\mat{A}\mat{B}^{-1}\,\bigr)}(\mat{A}\mat{B}^{-1}) =
\mat{I}$.  This means that $\mat{A} = \mat{C}\mat{B}$, where~$\mat{C}$
is any orthogonal matrix.

Now, we show that there are solutions to~\eqref{e3} of the form\\
$\mat{A} = \mat{I} + t\,\trsp{\mat{L}}\mat{L}$. Indeed, for such a
matrix~$\mat{A}$, we have
$$\trsp{\mat{L}}\mat{L}+\trsp{\mat{A}}\mat{A} - \mat{I} = \trsp{\mat{L}}\mat{L} +
\bigl(\mat{I}+t\,\trsp{\mat{L}}\mat{L}\bigr)^2-\mat{I} =
\bigl(1+2t+(1-a_0^2)t^2\bigr)\,\trsp{\mat{L}}\mat{L}.$$
Therefore $\mat{A}$ is a solution to~\eqref{e3} if and only if\quad
$t= -\frac{1}{1+\varepsilon a_0}$, where $\varepsilon=\pm1$. This
yields two matrices $\mat{A}_{1}$ and $\mat{A}_{-1}$. A simple
computation yields the corresponding~$\mat{V}_\varepsilon$:
$\mat{V}_\varepsilon = -\varepsilon\,\trsp{\mat{L}}$.

To summarize, we obtain two particular solutions to our problem:
\begin{equation}
  \mat{U}_\varepsilon =
  \begin{pmatrix}a_0&\mat{L}\\
    -\varepsilon\trsp{\mat{L}}& \mat{I}-(1+\varepsilon a_0)^{-1}\,\trsp{\mat{L}}\mat{L}
  \end{pmatrix}
\end{equation}
for $\varepsilon = \pm1$.\medskip

So the general solution is
\begin{equation}\label{general}
\begin{pmatrix}
  a_0&\mat{L}\\
  -a_0\mat{M}\mat{A}_1^{-1}\,\trsp{\mat{L}}& \mat{M}\mat{A}_1
\end{pmatrix}
\end{equation}
where $\mat{A}_1= \mat{I}-(1+a_0)^{-1}\,\trsp{\mat{L}}\mat{L}$ and
$\mat{M}$ is any orthogonal matrix.

\begin{remark}\label{iter}
  In Formula~\eqref{general}, we can use an orthogonal matrix~$\mat{M}$
  constructed in the same way as $\mat{A}_1$ by imposing its first
  row. This process can be iterated.
\end{remark}

\subsection{Orthogonal matrices with entries in a ring}\label{orth}

If the $a_j$ (for $0\le j< q$) are elements of a commutative
unitary ring~$\mathscr R$ such that $1+\varepsilon a_0$ is a unit of
$\mathscr R$ and $\sum_{0\le j< q} a_j^2=1$, the matrix
\begin{equation}\label{mat}
U_\varepsilon =
\begin{pmatrix}
  a_0&\mat{L}\\
  -\varepsilon \trsp{\mat{L}}& \mat{I}-(1+\varepsilon a_0)^{-1}\,\trsp{\mat{L}}\mat{L}
\end{pmatrix}
\end{equation}
satisfies $\trsp{\mat{U}_\varepsilon}\mat{U}_\varepsilon =\mat{I}.$\medskip

Fulfilling the relation $\sum_{0\le j< q} a_j^2=1$ is easy:
take $(b_1,b_2,\cdots,b_{q-1}) \in {\mathscr R}^{q-1}$ such that
$1+\sum_{1\le j<q} b_j^2$ is a unit, and set
\begin{equation}\label{stereo}
a_0 = \frac{1-\sum_{1\le j<q} b_j^2}{1+\sum_{1\le j<q} b_j^2} \text{ and } a_k = \frac{2b_k}{1+\sum_{1\le j<q} b_j^2} \text{ for } 1\le k<q.
\end{equation}
Then if 2 is a unit, $U_1$ is well defined, and $U_{-1}$ is defined if
$2\sum_{1\le j<q} b_j^2 $ is a unit. But the situation is a bit more
subtle as shown in the next section.

\subsubsection{The case
  ${\mathscr R}= {\mathbb Z}/256{\mathbb Z}$}\label{octets}\
\vspace*{1ex}

To construct $U_\varepsilon$ we take
$(b_1,b_2,\cdots,b_{q-1}) \in \{0,1,\dots,255\}^{q-1}$, such that
$\sum_{1\le j<q} b_j^2 \equiv 1 \mod 4$. Then the fractions
in~\eqref{stereo}, once reduced, have odd denominators and the
numerator of $a_0$ is even. This means that Formulas~\eqref{mat}
and~\eqref{stereo} make sense in ${\mathbb Z}/256{\mathbb Z}$.

When $q$ runs from 2 to 13, the numbers $\nu(q)$ of possible vectors
$(b_1,b_2,\cdots,b_{q-1})$ are respectively\\
$2^7, 2^{15}, 3\cdot 2^{21}, 2^{30}, 3\cdot 2^{36}, 3\cdot 2^{44},
7\cdot 2^{51}, 2^{62}, 17\cdot 2^{66}, 17\cdot 2^{74}, 3\cdot 11\cdot
2^{81},2^{94}$.

Note that each vector gives rise to two matrices. These
numbers~$\nu(q)$ were calculated using the Maple computer algebra
system, via the following algorithm. For $m\ge 1$ and $j=0,1,2,3$, let
$$E_m(j) = \left\{x=\sum_{1\le k\le m} y_k^2\ :\ x\equiv j \mod 4.\ \  y_k \in {\mathbb Z}/256{\mathbb Z}\right\},
$$
and for $x\in \bigcup_{j=0}^3 E_m(j)$ let $N_m(x)$ stand for the
number of ways of writing $x=\sum_{1\le k\le m} y_k^2$. If
$x_1\in E_1(j_1)$ and $x_2\in E_m(j_2)$, then
$x = x_1+x_2\in E_{m+1}(j)$ whenever $j\equiv j_1+j_2 \mod 4$, and
$N_1(x_1)N_m(x_2)$ contributes to the computation of~$N_{m+1}(x)$.
Finally, $\nu(q)=N_{q-1}(1)$. We have not further investigated a
formula for~$\nu(q)$.

\section{Some symmetric cryptosystems}

\subsection{A first symmetric cryptosystem}

We wish to take advantage of the fact that computers handle
arithmetic modulo 256 very efficiently. So, we use the
ring~${\mathscr R}= {\mathbb Z}/256{\mathbb Z}$.\medskip

Positive integers~$n$ and~$q\ge 2$ will be chosen later. For
each~$j\in\{1,2,\cdots,n\}$, we choose at
random~$\varepsilon_j\in \{0,1\}$ and $q-1$ numbers modulo 256,
$(b_i^j)_{1\le i< q}$, fulfilling the requirements of
paragraph~\ref{octets}. This set of $nq$ numbers constitutes the key,
whose size is therefore $n(8q-7)$ bits. The number of such keys is
$\bigl(2\nu(q)\bigr)^n$.

From this key, we can construct, according to Section~\ref{orth},
$n$~orthogonal matrices, $\mat{R}_1,\mat{R}_2,\cdots,\mat{R}_{n}$, of
dimension~$q$, with coefficients in ${\mathbb Z}/256{\mathbb Z}$. The
plaintext to be encrypted is a block of $q^n$ bytes, i.e., a
vector~$X$ in $\left({\mathbb Z}/256{\mathbb Z}\right)^{q^n}$. The
ciphertext is
$\mat{R}_1\otimes \mat{R}_2\otimes\cdots\otimes \mat{R}_nX$.
Encryption and decryption are both performed using FTP. Both
encryption and decryption involve $nq^n$ byte multiplications, i.e.,
$M\log_q M$, where~$M$ stands for the length of the message.\bigskip

Consider the following table to choose an implementation of this
cryptosystem meeting the desired requirements in terms of security and
speed.

\begin{center}\tiny
$\begin{array}{cccccc}
0 & 2 & 3 & 4 & 5 & 6 
\\
 2 & 16,18,5,3 & 24,27,6,5 & 32,36,7,6 & 40,45,8,8 & 48,54,9,9 
\\
 3 & 32,34,7,5 & 48,51,8,7 & 64,68,10,9 & 80,85,11,11 & 96,102,13,13 
\\
 4 & 46,50,7,5 & 69,75,9,8 & 92,100,11,10 & 115,125,13,13 & 138,150,15,15 
\\
 5 & 62,66,8,6 & 93,99,10,9 & 124,132,13,12 & 155,165,15,14 & 186,198,17,17 
\\
 6 & 76,82,9,7 & 114,123,11,10 & 152,164,14,13 & 190,205,16,16 & 228,246,19,19 
\\
 7 & 92,98,9,7 & 138,147,12,11 & 184,196,15,14 & 230,245,18,17 & 276,294,20,20 
\\
 8 & 108,114,9,7 & 162,171,12,11 & 216,228,15,14 & 270,285,18,18 & 324,342,21,21 
\\
 9 & 126,130,10,8 & 189,195,13,12 & 252,260,16,15 & 315,325,19,19 & 378,390,23,22 
\\
 10 & 142,146,10,8 & 213,219,13,12 & 284,292,17,16 & 355,365,20,19 & 426,438,23,23 
\\
 11 & 158,162,10,8 & 237,243,14,12 & 316,324,17,16 & 395,405,21,20 & 474,486,24,24 
\\
 12 & 174,178,11,9 & 261,267,14,13 & 348,356,18,17 & 435,445,21,21 & 522,534,25,25 
\end{array}
$
\end{center}
Rows are numbered according to~$q$ ($2\le q\le 12$), and columns
according to~$n$ ($2\le n\le 6$). Each entry has four terms:
\begin{itemize}
\item [1 --] $\log_2$ of the number of keys: $n\log_2(2\nu(q))$,
  rounded to the largest integer below,
\item [2 --] the key length: $n(8q-7)$ bits,
\item [3 --] $\log_2$ of the block size (bits): $\log_2 8q^n$, rounded to the smallest integer above,
\item [4 --] $\log_2$ of the number of byte multiplications for
  encryption or decryption: $\log_2(nq^n)$ rounded to the smallest
  integer above.
\end{itemize}

It should be noted that this system is somewhat vulnerable to
chosen-plaintext attacks~\cite{menezes}, due to the tensor structure of
$\mat{R} = \mat{R}_1\otimes \mat{R}_2\otimes\cdots\otimes \mat{R}_n$,
and the fact that a block is always encoded in the same way.  For
instance, the knowledge of the first $q$ columns of $\mat{R}$
unveils~$\mat{R}_1$ and the first column of
$\mat{R}_2\otimes \mat{R}_3\otimes\cdots\otimes \mat{R}_n$.

We must therefore eliminate this vulnerability while thwarting
differential cryptanalysis~\cite{biham}. This is the subject of the
next section.

\subsection{A better cryptosystem}

In this section, $q$, $n$, $\nu$, and $\mat{R}$ have the same meanings
as previously. The addition in the ring
${\mathbb Z}/256{\mathbb Z}^{q^n}$ is denoted by
$\oplus$; this is the addition modulo 256 component by component.

To secure the system we use the CBC technique. Let
$c_1\in ({\mathbb Z}/256{\mathbb Z})^{q^n}$ be the plaintext to be
encrypted.

We choose at random an element $c_0$ in
$({\mathbb Z}/256{\mathbb Z})^{q^n}$. Then the encrypted message is
the concatenation of $e_0 = \mat{R}c_0$ and of
$e_1 = \mat{R}(c_1\oplus e_0)$. With this procedure we keep
the same key, but the ciphertext is two blocks instead of one. If we
have one more block $c_2$ to encrypt, we just add
$e_2 = \mat{R}(c_2\oplus e_1)$, and so on if there are more
blocks. Obviously, the more blocks there are, the smaller the relative
overhead becomes.

An alternative is to incorporate $c_0$ into the key and then transmit
only $\mat{R}(c_1\oplus\mat{R}c_0)$.

\subsubsection*{\textbf{\textsl{ Three examples of implementation}}}

\begin{enumerate}

\item $q=4$ and $n=2$. Then the block size is $16$ bytes and the
  key length is $50$ bits. A brute-force attack requires up to
  $\bigl(2\nu(4)\bigr)^2\cdot 256^{16}= 3^2 2^{44+8\times16} = 3^2 2^{172}$
  trials.  The cost of encryption or decryption of~$m$ blocks is
  $32(m+1)$ multiplications and $16m$ additions of bytes.

\item This is a variant of the preceding. Again $q=4$, $n=2$, and the
  block length is 16 bytes, but this time we use
  Remark~\ref{iter}. Then the length of the key is
  $2(8q-7)+2\left(8(q-1)-7\right) = 84$ bits, and the brute-force
  attack requires up to
  $\bigl(4\nu(3)\nu(4)\bigr)^2\cdot 256^{16}= 3^2 2^{204}$ trials.
  
\item $q=4$ and $n=3$. This leads to blocks of $64$ bytes, and to a
  key length of 75 bits. A brute-force attack requires up to
  $\bigl(2\nu(4)\bigr)^3\cdot 256^{64}= 3^3\,2^{578}$ trials. The cost
  of encryption or decryption of~$m$ blocks is $192(m+1)$
  multiplications and $64m$ additions of bytes.
\end{enumerate}

\end{document}